\documentclass{amsart}
\title[Fatou-Bieberbach domains ]{Fatou-Bieberbach domains as basins of attraction of Automorphisms tangent to the Identity}
\author{Liz Raquel Vivas}
\date{July 6, 2009}

\usepackage{amscd}

\newtheorem{theorem}{Theorem}
\newtheorem{lemma}{Lemma}
\newtheorem{proposition}{Proposition}
\newtheorem{corollary}{Corollary}

\theoremstyle{definition}
\newtheorem{defin}{Definition}

\theoremstyle{remark}
\newtheorem{remark}{Remark}

\newcommand{\NN}{\mathbb{N}}

\newcommand{\CC}{\mathbb{C}}
\newcommand{\PP}{\mathbb{P}}

\newcommand{\R}{\mathcal{R}}

\newcommand{\Id}{\mathrm{Id}}
\newcommand{\Arg}{\textrm{Arg}}
\newcommand{\xra}{\xrightarrow}

\def\e{{\epsilon}}

\begin{document}

\bibliographystyle{plain}

\begin{abstract}
We prove that there exists an automorphism of $\CC^2$ tangent to the identity with a domain of attraction to the origin, biholomorphic to the origin, along a degenerate characteristic direction. 
\end{abstract}

\maketitle

\section{Introduction}

In the 1920's, Fatou and Bieberbach proved the existence of proper subdomains of $\CC^2$ that are biholomorphically equivalent to $\CC^2$, nowadays known as Fatou-Bieberbach domains. Their examples were the basins of attraction of automorphisms of $\CC^2$ with more than one fixed point. In fact, the basin of attraction of an \textit{attracting} fixed point of an automorphism of $\CC^k$ is always biholomorphic to $\CC^k$, with the fixed point in the interior of the domain. A complete proof of this fact, as well as a variety of examples, was given by Rosay and Rudin \cite{ro-ru}. In the same article the authors posed several questions about Fatou-Bieberbach domains. In recent years several mathematicians have made progress towards solving some of these questions (see for example work by Peters and Wold\cite{pe-wo}, Peters \cite{pe1}, Stens\o nes \cite{st}, Wold \cite{wo1}, \cite{wo2}). 

One of the most interesting remaining open questions is the following: Does there exist a Fatou-Bieberbach domain in $\CC^2$ that avoids the set $\{zw = 0\}$? In the aforementioned paper, Rosay and Rudin proved that it is possible to avoid one complex line. It has also been proved by  Green \cite{gr}, that a non-degenerate image of $\CC^2$ cannot avoid three complex lines. The answer for the case of two lines still remains open. In this paper we find a Fatou-Bieberbach domain in $\CC^2$ that avoids one axis and one image of $\CC$ in $\CC^2$ that is locally tangent to an arbitrarily high order to the other axis (see Theorem \ref{degbasinexample}). We will make these statements more precise after we introduce the definitions at the end of this section.

As a step towards finding (counter)examples for open questions about Fatou-Bieberbach domains, it is natural to look for methods to construct such domains. In 1997 Weickert \cite{we}, and then Hakim \cite{hak1} found an entire new class of Fatou-Bieberbach domains. The construction goes as follows: Let $F$ be an automorphism of $\CC^k$ tangent to the identity at a fixed point $p$ (i.e. $F(p) = p$ and $DF(p) = \Id$). Under some hypotheses on the 2-jet of the automorphism there will exist a \textit{non-degenerate characteristic direction} $[v] \in \PP^{k-1}$ (we should think of $v$ as a vector in $\CC^k$ and of $[v]$ as the natural image in $\PP^{k-1}$). Define the basin of attraction of $p$ along $[v]$ as follows:
$$\Omega_{p,[v]} = \{z \in \CC^k \mid \lim_{n \to \infty}f^n(z) = p;  \lim_{n \to \infty}[f^n(z)] = [v] \}.$$
Then depending on a certain index $r_{[v]}$ associated to this direction we have that $\Omega$ is biholomorphic to $\CC^k$, and $p$ is in the boundary of $\Omega$. 

One strategy to find a Fatou-Bieberbach domain that avoids two lines is to consider an automorphism of $\CC^2$ for which the origin is a neutral fixed point and restricts to the identity on the coordinate axes. Then, if we could satisfy the extra assumptions needed in Hakim's theorem we could find a Fatou-Bieberbach domain attracted to the origin that avoid both axes. In \cite{vi}, we prove that, assuming a well-known conjecture, this will never be the case. In particular we prove that, under the assumptions of Hakim's theorem the number $r_{[v]}$ is always negative and therefore there will not be an open domain attracted to the origin.

Nonetheless, in this paper we prove that there exist automorphisms of $\CC^2$ tangent to the indentity at a fixed point admitting basins of attraction along \textit{degenerate} characteristic directions. These automorphisms, however, do not preserve both axes.

The local dynamics of maps tangent to the identity at a fixed point have been studied with a lot of success in the last years by Weickert \cite{we}, Hakim \cite{hak1}, Abate \cite{ab1} and Abate-Bracci-Tovena \cite{ab-br-to}. Our work is also a contribution in this direction, because it features dynamical behavior that has not been seen before. 

From now on, all the maps we consider have the origin as the fixed point, and when we say tangent to identity we mean it is tangent to the identity at the origin. Then Weickert proved \cite[Thm. 1]{we}:

\begin{theorem}\label{wetheom}
There exist automorphisms of $\CC^2$ tangent to the identity with an invariant domain of attraction to the origin, biholomorphic to $\CC^2$, on which the automorphism is biholomorphically conjugate to the map
\[
(x,y) \to (x-1,y)
\]
\end{theorem}

Hakim \cite{hak1} proved a more general result. In order to state her theorem, we need to introduce some definitions.

\begin{defin}\label{defin:chardir}
A \textit{parabolic curve} or an \textit{invariant piece of curve} for $F$ at the origin is an injective holomorphic map $\phi: \Delta \to \CC^n$ satisfying the following properties:
\begin{itemize}
\item $\Delta$ is a simply connected domain in $\CC$ with $0 \in \partial \Delta$;
\item $\phi$ is continuous at the origin, and $\phi(0)= O$; 
\item $\phi(\Delta)$ is invariant under $F$, and $\left(F^n|_{\phi(\Delta)}\right) \to O$ as $n \to \infty$.
\end{itemize}
Furthermore, if $[\phi(\zeta)] \to [v] \in \PP^{n-1}$ as $\zeta \to 0$ (where $[.]$ is the projection of $\CC^n \backslash \{O\}$ onto $\PP^{n-1}$), we say that \textit{$\phi$ is tangent to $[v]$ at the origin}. Writing 
\begin{equation*}
F(z) = z + P_k(z) + P_{k+1}(z) + ...
\end{equation*}
where for each $h \in \NN, h \geq k$, $P_h$ is a homogeneous polynomial function of degree $h$ from $\CC^n$ to $\CC^n$, and $P_k \not \equiv 0$. Then $k = \nu(F)$ is the \emph{order} of $F$.

A \textit{characteristic direction} for $F$ is a vector $[v] \in \PP^{n-1}$ such that there is $\lambda \in \CC$, so that $P_k(v) = \lambda v$. If $\lambda \neq 0$, we say that $v$ is \textit{non-degenerate}; otherwise, it is \textit{degenerate}. Note that $[v]$ is a non-degenerate characteristic direction if the induced map of $P_k$ on $\PP^{n-1}$ has $[v]$ as a fixed point.
\end{defin}

Then Hakim proved the following result \cite[Thm. 1.3]{hak1}:

\begin{theorem}\label{nondegcurve}
Let $F$ be a germ of an analytic transformation from $\CC^k$ fixing the origin and tangent to the identity. For every nondegenerate characteristic direction $[v]$ of $F$, there exists a parabolic curve, tangent to $[v]$ at zero, attracted to the origin. 
\end{theorem}

In order to have \emph{open} regions attracted to the origin, though, we need to look at one more invariant associated to the non-degenerate characteristic direction. We follow Abate's exposition \cite{ab2} in order to define this invariant, because it is more illustrative than the original definition:

\begin{defin}\label{director}
Given $F$ a germ of an analytic transformation of $\CC^k$ fixing the origin and tangent to the identity, and a non-degenerate characteristic direction $[v] \in \PP^{k-1}$, the eigenvalues $\alpha_1, \alpha_2, \dots, \alpha_{k-1} \in \CC$ of the linear operator $D(P_k)_{[v]} - \Id: T_{[v]}\PP^{k-1} \to T_{[v]}\PP^{k-1}$ are the \textit{directors} of $[v]$. 
\end{defin}

Hakim also proves \cite{hak1}:

\begin{theorem}\label{haktheo}
Let $F$ be an automorphism of $\CC^2$ tangent to the identity. Let $[v]$ be a non-degenerate characteristic direction at $0$. Assume that the real parts of all the directors of $[v]$ are strictly positive. Let $\Omega$ be defined by
\[
\Omega = \{ z \in \CC^k \backslash \{0\}: \lim_{n \to \infty} z_n = 0, \lim_{n \to \infty} [z_n] = [v] \}
\]
Then $\Omega$ is biholomorphic to $\CC^2$. In this domain the automorphism is biholomorphically conjugate to the map
\[
(x,y) \to (x-1,y)
\]
\end{theorem}

It turns out that the domain of Weickert's example is the set $\Omega_{(0,V)}$ for a non-degenerate characteristic direction $V$ associated to a positive director.

Although these are very interesting results, they do not say anything for the case of degenerate characteristic directions. In this paper we prove:

\begin{theorem}\label{degbasinexample}
There exists an automorphism $F$ of $\CC^2$ tangent to the identity such that there is an invariant domain $\Omega$ in which every point is attracted to the origin along a trajectory tangent to $v$, where $v$ is a degenerate characteristic direction of $F$ on which the automorphism is biholomorphically conjugate to the map
\[
(x,y) \to (x-1,y).
\]
Moreover, $\Omega$ is a Fatou-Bieberbach Domain and $\Omega \cap \{(0,w) : w\in \CC\} = \emptyset$. Furthermore, there is a biholomorphic copy of $\CC$ injected in $\CC^2$, locally tangent to the $z$-axis, that is entirely contained in the boundary of $\Omega$.
\end{theorem}

The proof of the theorem is similar in spirit to the proof of Weickert's Theorem (Theorem~\ref{wetheom}). We first explain the general structure of both proofs. In the remarks that follow we point out the major differences between the case of non-degenerate characteristic directions and that of degenerate directions.

Given an automorphism $F$ of $\CC^2$ tangent to the identity at the origin. The steps we follow to show that the region of attraction is a Fatou-Bieberbach domain are:

\begin{enumerate}

\item First we find a domain $D$ such that $F(D) \subset D$ and $0 \in \partial D$. Also we prove that for any $z \in D, F^n(z) \to 0$. Moreover, $[F^{n}(z)] \to [v]$ where $[v] \in \PP^1$ is a \textit{degenerate} characteristic direction.

\item We work locally in $D$ and find an Abel-Fatou coordinate. That is, $\phi: D \to \CC$ such that $\phi (F(p)) = \phi(p) - 1$.

\item A global basin of attraction $\Omega$ is obtained as $\Omega = \cup_{n=0}^{\infty} F^{-n}(D)$. We extend $\phi$ to all of $\Omega$. We prove that $\phi$ is surjective onto $\CC$ (this normally comes for free with the use of the Abel-Fatou coordinates). 

\item After this we prove that $V_t = \phi^{-1}(t)$ is biholomorphic to $\CC$ for each $t$. We call the biholomorphism $\psi_t: V_t \to \CC$.

\item Using $\phi$ and $\psi_t$, we construct a global map $G = (\phi,\psi): \Omega \to \CC^2$. We prove $G$ is a biholomorphism, i.e. $G$ is injective and surjective.
\end{enumerate}

Note that many of these steps could also be carried out for a germ of a biholomorphism. Therefore we find explicit examples of germs for which there exists an open basin along a degenerate characteristic direction. This phenomenon has already been explored by Abate in \cite{ab2}, where he gives an example of a germ of $\CC^2$ with $(0,1)$ a degenerate characteristic direction and an open basin along $(0,1)$ (see case $1_{(10)}$ Page 8).

The major difference between the case of degenerate and non-degenerate characteristic direction is in Step 2. The condition of being a non-degenerate characteristic direction allows one to make a simple change coordinates in order to prove Step 2. In the general case, this change of coordinates does not always exist. In our case we prove the existence of $\phi$ by solving a differential equation. We remark that given Steps 1 through 4, Step 5 is exactly the same as in Weickert's case. 

We should point out that for most of these steps to be completed, we need a very good estimate of the size of the change of variables. The difficulty of proving surjectivity depends in general on how the shape of our domain changes with the change of coordinate. 

The presentation of the proof is as follows. We start with an automorphism of $\CC^2$ tangent to the identity. Then each section will be used to prove steps 1 through 5 of the outline of the proof.  We also prove the existence of an $F$-invariant curve $\Gamma$ biholomorphic to $\CC$ and contained in $\partial \Omega$, where $\Omega$ is the Fatou-Bieberbach domain. 

\medskip

\noindent
{\it Acknowledgments.} I would like to thank my advisor Prof. Berit Stensones for fruitful discussions and Prof. Mattias Jonsson for invaluable comments on an earlier draft.

\section{Invariant Domain}
We first describe a class of automorphisms where we are able to find a basin.

\begin{theorem} Let $F$ be an automorphism tangent to the identity at the origin of the form:
\begin{align}
F(z,w) = (z + z^2 + O(z^3,z^2w), w - zw^2 + O(z^4, z^3w, z^2w^2, zw^3)). \label{especial}
\end{align}
Then the vector $(0,1)$ is a degenerate characteristic direction for $F$. Moreover, there exists a basin of attraction $D$, for which every point $p \in D$ is attracted to the origin along $(0,1)$.
\end{theorem}

It is not obvious that there exist automorphisms of $\CC^2$ that have a germ as in \eqref{especial}. Buzzard-Forstneric \cite{bu-fo} and Weickert \cite{we2} have proven that given a finite jet germ of a biholomorphic map, it is possible to construct a global automorphism with the prescribed jet. However, this theorem is not enough for our purposes, since we want to prescribe \emph{which} higher order terms can appear. Specifically we want $F(0,w) = (0,w)$.

Below we construct automorphisms of the form \eqref{especial} using a composition of shears and overshears.

\begin{lemma}\label{lemma:constructionofF}
There exists $F \in Aut(\CC^2)$ such that:
\begin{align*}
F(0)&= 0\\
F'(0)&= \Id
\end{align*}
and
\begin{equation}
F(z,w) = (z + z^2 + O(z^3,z^2w), w - w^2z - \frac{z^4}{3} + \frac{8z^3w}{3} + O(z^5,z^4w,z^2w^2,zw^3))
\end{equation}
is the germ of F at (0,0).
\end{lemma}

\begin{proof}
We will define $F$ as the compositions of shears and overshears in $\CC^{2}$. Let us start with: 
\begin{eqnarray*}
(z,w) &\xra{f_{1}} &(z,w+z)\\ 
(z,w) &\xra{f_{2}} &(ze^w,w)\\ 
(z,w) &\xra{f_{3}} &(z,w - z) \\
(z,w) & \xra{f_{4}} &(ze^{-w},w)
\end{eqnarray*}

If we compute the power series of $f = f_4\circ f_3 \circ f_2 \circ f_1$ around the origin we obtain:
\begin{equation*}
f(z,w) = (z + z^2 + O(z^3, z^2w), w - zw - z^2 - \frac{z^3}{2} - z^2w - \frac{z^3}{2} + O(z^4, z^3w, z^2w^2, zw^3))
\end{equation*}

We want to get rid of the terms $-zw, -z^{2}, -z^2w$ and $-\frac{z^{3}}{2}$ in the second coordinate. So, we apply  a shear and an overshear to cancel these terms. 

The terms $z^{\alpha}$ can be canceled by using shears of the kind $(z,w) \to (z, w + f(z))$ and the terms $z^{\beta}w$ can be canceled using overshears of the kind $(z,w) \to (z, we^{g(z)})$.

We call $O(4) = O(|z|^4, |z|^3|w|, |z|^2|w|^2, |z||w|^3)$.

\begin{itemize}
\item To cancel $-z^2$;
we use the shear $s_1(z,w) = (z,w + z^2)$. Now we compute the power series expansion of $s_1 \circ f$  around $(0,0)$ and we get:
\begin{eqnarray*}
s_1 \circ f (z,w) = (z+z^2 + O(z^3, z^2w), w + \frac{3z^3}{2} - zw - z^2w - \frac{zw^2}{2} + O(4))
\end{eqnarray*}

\item To cancel $\frac{3z^3}{2}$;
we use the shear $s_2(z,w) = (z, w - \frac{3z^3}{2})$ and now we get:
\begin{eqnarray*}
g(z,w) = s_2 \circ s_1 \circ f (z,w) = (z+z^2 + O(z^3, z^2w), w  - zw - z^2w - \frac{zw^2}{2} + O(4))
\end{eqnarray*}

\item To cancel $-zw$;
we use the overshear $o_1(z,w) = (z, we^{z})$ and now we get:
\begin{eqnarray*}
o_1 \circ g(z,w) = (z+z^2 + O(z^3, z^2w), w- \frac{z^2w}{2} - \frac{zw^2}{2} + O(4))
\end{eqnarray*}

\item To cancel $\frac{-z^2w}{2}$;
we use the overshear $o_2(z,w) = (z, w e^{z^2/2})$ and now we get:
\begin{eqnarray*}
h(z,w) = o_2 \circ o_1 \circ g (z,w) = (z+z^2 + O(z^3, z^2w), w  - \frac{zw^2}{2} + O(4))
\end{eqnarray*}
\end{itemize}
which is the desired power series around $0$. 

We can also conjugate by $\phi (z,w) = (z, 2w)$ and finally we have:
\begin{eqnarray*}
F(z,w) = \phi^{-1} \circ h \circ \phi(z,w) = (z+z^2 + O(z^3, z^2w), w  - zw^2 + O(z^4, z^3w, z^2w^2, zw^3))
\end{eqnarray*}

To summarize we are composing the following shears and overshears:
\begin{eqnarray*}
(z,w) &\xra{\phi} & (z,2w)\\
(z,w) &\xra{f_{1}} &(z,w+z)\\ 
(z,w) &\xra{f_{2}} &(ze^w,w)\\ 
(z,w) &\xra{f_{3}} &(z,w - z) \\
(z,w) & \xra{f_{4}} &(ze^{-w},w)\\
(z,w) &\xra{s_{1}} &(z,w+z^2)\\ 
(z,w) &\xra{s_{2}} &(z,w - \frac{3z^3}{2})\\
(z,w) &\xra{o_{1}} &(z,we^{z})\\
(z,w) &\xra{o_{2}} &(z,we^{z^2/2})\\   
(z,w) &\xra{\phi^{-1}} &(z, w/2)
\end{eqnarray*}
and
\[
F = \phi^{-1} \circ o_2 \circ o_1 \circ s_2 \circ s_1 \circ f_4 \circ f_3 \circ f_2 \circ f_1 \circ \phi
\]
so 
\begin{eqnarray*}
F(z,w) = (F_1(z,w), F_2(z,w)) 
\end{eqnarray*}
where
\begin{align}
\label{Fzcoordinates} F_1(z,w) &= z \exp(ze^{2w+z}) 
\end{align}
and
\begin{align}
\label{Fwcoordinates} F_2(z,w) =& \left[ w+ \frac{z}{2} - \frac{z}{2}e^{2w+z} + \frac{z^2}{2}\exp(2ze^{2w+z}) - \frac{3z^3}{4}\exp(3ze^{2w+z})\right] \times\\ \nonumber
& \exp \left(z\exp(ze^{2w+z}) + \frac{z^2}{2}\exp(2ze^{2w+z})\right) 
\end{align}
which gives the expansion we were looking for.
\end{proof}

Notice that by using more shears and overshears we can get rid of the pure terms $z^{\alpha}$ and $z^{\beta}w$ for $\alpha$ and $\beta$ as large as we want meaning we can get an automorphism of $\CC^2$ with power series around $(0,0)$:
\smallskip
$$
F(z,w) = (z+z^2+O(z^3,z^2w), w - w^2z +O(z^\alpha, z^{\beta}w, z^2w^2, zw^3))
$$

Since $F(0,w) = (0,w)$ it follows that $F$ restricted to $\CC^* \times\CC$ is an automorphism of $\CC^* \times\CC$. 

As we said before, it is easy to see that $F$ has two characteristic directions: $(1,0)$, which is nondegenerate and $(0,1)$, which is degenerate. 

Hakim's Theorem~\ref{nondegcurve} says that for any nondegenerate characteristic direction $[v]$, there exists an invariant piece of curve attracted to the origin tangent to $[v]$. Applying this theorem in our setup, it follows there exists an invariant piece of curve attracted to $(1,0)$ for the map $F$ in Lemma \ref{lemma:constructionofF}. 

The assumption of the existence of a non-degenerate characteristic direction allows us to blow-up the origin to get a simpler expression for $F$. In these new coordinates Hakim proves that the invariant piece of curve is locally a holomorphic graph over the $z$-axis, and she gives precise estimates of the size of the graph function and its derivative. We will also need estimates in our case therefore we present her results and apply them to our case. The estimates of her paper that we use are scattered in different parts of the paper. Therefore for the reader's convenience we summarize them in the following proposition  (see Main Theorem 1.3, Lemma 4.5, Lemma 4.6 and Proposition 4.8 in \cite{hak1}).

\begin{proposition}\label{mihakim} Assume that a transformation tangent to the identity is written as follows:
\begin{align}
x_1 &= f(x,u) = x - x^2 + O(ux^2, x^3),\label{hakimmap}\\
u_1 &= \Psi(x,u) = u - axu + O(u^2x,ux^2) + x^{k+1} \psi_k(x).\nonumber 
\end{align}
where $a \notin \NN$. Then there exists an invariant piece of curve $(x,u(x))$, where $u$ is defined in some $D_{r} = \{x \in \CC; |x-r| < r\}$ and where $\lim_{x \in D_r, x \to 0}u(x)=0$. Moreover,
\begin{align}\label{hakimusize}
|u(x)| \leq C_1|x|^k  \quad \textrm{and}\quad  |u'(x)| \leq C_2|x|^{k-1}
\end{align}
for $x\in D_r$, where $C_1$ and $C_2$ are positive constants.
\end{proposition}

We refer to Hakim's proof of Main Theorem 1.3 for this result, but let us point out that we should compare \eqref{hakimmap} with equation (4.1) on section 4 of \cite{hak1}. Note that, as stated, Proposition~\ref{mihakim} does not involve $\log$ terms as opposed to equation (4.1) in \cite{hak1}. This is because we are assuming $a\notin \NN$. The estimates we are quoting for the size of $u$ should be compared to Lemmas 4.5 (p.420) and 4.6 (p.421) of \cite{hak1}. In our case we do not need the $\log$ terms since $a \notin \NN$, and we will use $k=4$ in both lemmas.

Let us define, for any $\epsilon > 0$:
\begin{equation}
V_{\epsilon} = \{ z \in \CC \mid  0<|z| < \epsilon , |\textrm{Arg}(z) - \pi| < \pi/8\} \subset \CC
\end{equation}

It is clear that for $\epsilon$ small enough, $z \in V_{\epsilon}$ implies $-z \in D_r$. Note that $V_{\epsilon}$ is a local basin of attraction for the map $z \mapsto z + z^2$.

Now we can give our result:
\begin{proposition}\label{prop:gamma} Let $F$ be as in Lemma \ref{lemma:constructionofF}. Then there exists a parabolic curve attracted to $(0,0)$ along $(1,0)$ i.e. there exists $\gamma$ defined in the closure of $V_{\epsilon}$ with values in $\CC$ such that $\gamma$ is analytic in $V_\epsilon$, continuous up to the closure of $V_{\epsilon}$, $\gamma(0) = 0$ and such that if $p \in \mathcal{C} = \{(z,\gamma(z)), z \in V_{\epsilon}\}$ then $F(p) \in \mathcal{C}$ and $F^{n}(p) \rightarrow 0$ when $n \rightarrow \infty$.
Also we have the following estimates on the size of $\gamma$ and its derivative, i.e. there exist constants $C_1$ and $C_2$ such that:
\begin{eqnarray}
\label{gammasize}|\gamma(z)| \leq C_1|z|^3
\end{eqnarray}
and
\begin{eqnarray}
\label{gammaprimesize}|\gamma'(z)| \leq C_2|z|^2
\end{eqnarray}
for $z \in V_{\epsilon}$.
\end{proposition}

\begin{proof}

This result follows almost directly from Proposition~\ref{mihakim}. The existence of an attracting piece of curve is immediate, since $(1,0)$ is a non-degenerate characteristic direction. 
For the estimates on the size of the graph, we want to change coordinates so that 
\begin{eqnarray*}
F(z,w) = (z + z^2 + O(z^3, z^2w), w - w^2z +O(z^4, z^3w, z^2w^2, zw^3))
\end{eqnarray*}
is locally of the form \eqref{hakimmap}.

Let our invariant curve for our original map be $\mathcal{C} = \{(z,\gamma(z)), z \in V_{\epsilon}\}$. We know this curve exists by Hakim's theorem, since $(1,0)$ is a non-degenerate characteristic direction.

We use the change of coordinates:
\begin{eqnarray*}
(x,u) := \varphi(z, w) = (-z, w)\\
(z,w) = \varphi^{-1}(x,y) = (-x, u)
\end{eqnarray*}
In these coordinates we have:
\begin{eqnarray*}
\tilde{F} (x,y) &=& \varphi \circ F \circ \varphi^{-1} (x,y) = \varphi \circ F (-x, y)\\
&=& \varphi (-x + x^2 + O(x^3, x^2u), u + xu^2 + cx^4  + O(x^5, x^3u, x^2u^2, xu^3))\\
&=& (x-x^2+O(x^3,x^2u), u + O(x^3u,xu^2) + x^4\psi(x))
\end{eqnarray*}

Comparing with \eqref{hakimmap}, we can see our map is in the desired form, for $a=0$ and $k=4$. Under the change of coordinate $\varphi$ the invariant curve $(z,\gamma(z))$ will be $(x, u(x))$ where:
\begin{eqnarray*}
(x, u(x)) &=& \varphi(z, \gamma(z))
\end{eqnarray*}
and $(x, u(x))$ is the attractive invariant curve for $\tilde{F}$. We use the estimates we have in \eqref{hakimusize} and the relationship:
\begin{eqnarray*}
x &=& -z\\
u(x) &=& \gamma(z)
\end{eqnarray*}

And we obtain:
\begin{eqnarray*}
|\gamma(z)| \leq C_1|z|^3\\
|\gamma'(z)| \leq C_2|z|^2
\end{eqnarray*}
for $z\in V_{\epsilon}$.
\end{proof}

We can compute the director associated to $v = (1,0)$. We use the following  equivalent definition as in Definition \ref{director} of the director (see \cite{ab1} for a proof of this fact):
\begin{align*}
A(v) :=\frac{r'(u_0)}{P_k(1,u_0)} 
\end{align*}
in our case we have $P_2(z,w) = z^2$, $Q_2(z,w) = 0$. We defined $r(u) = Q_2(1,u) - uP_2(1,u)$, then $r(u) = 0 - u = -u$, and $r'(u_0) = -1$. We have then $A((1,0)) = -1$. Therefore we can expect there is not an open invariant region around $\mathcal{C}$. 

We can extend the curve $\mathcal{C}$ and prove:

\begin{proposition}
If we define 
\begin{align}\label{invariantbigcurve}
\displaystyle\Gamma = \bigcup_{n\geq 0} F^{-n}(\mathcal{C})
\end{align}
 then we have that \begin{itemize}
\item [(i)]
$F(\Gamma)=\Gamma$
\item [(ii)]
$\Gamma$ is biholomorphic to $\CC$. 
\end{itemize}
\end{proposition}

\begin{proof}
The first part is clear. For the second part we need to understand the action of $F$ in $\mathcal{C}$. We have $\mathcal{C} = \{(z,\gamma(z)) ; z \in V_{\epsilon} \}$, with $F(\mathcal{C}) \subset \mathcal{C}$.
Clearly this action is conjugated to the following action:
$$
z_1 = z + z^2 + O(z^3,z^2w) = z + z^2 + O(z^3) 
$$
on $V_{\epsilon}$. It is a consequence of Fatou's work \cite{fa2} that the transformation is conjugated to $\zeta_1 = \zeta+1$ in a domain of type $U = \{\Re e \zeta > R\}$, for $R$ big enough. 
Define the following holomorphic map from $\Gamma$ to $\CC$: 
\begin{align*}
\Phi:& \Gamma \to \CC\\
\Phi(p) &= \zeta(\pi_1 \circ F^N(p)) - N
\end{align*}
where $\pi_1$ is the projection in the first coordinate and $N$ is large enough so $F^N(p) \in \mathcal{C}$. We check now that $\Phi$ is well-defined, injective and surjective. 
\begin{itemize}
\item $\Phi$ is well defined. We want to prove that $\Phi$ is independent of $N$. If $N$ and $M$ are both integers such that $F^N(p)$ and $F^M(p)$ are both in $\mathcal{C}$, then without loss of generalization we can assume $N < M$. Then $F^{M-N}(F^N(p)) = F^M(p)$ implies $\zeta(\pi_1(F^M(p))) = \zeta(\pi_1(F^N(p))) + M-N$. Therefore $\zeta(\pi_1(F^M(p))) - M = \zeta(\pi_1(F^N(p))) - N$, and we see $\Phi(p)$ is well-defined.
\item $\Phi$ is injective. Let $p$ and $q$ in $\Gamma$ such that $\Phi(p) = \Phi(q)$. We know there exists $N$ large enough such that $F^N(p)$ and $F^N(q)$ are both in $\mathcal{C}$. Therefore we will have $\zeta(\pi_1 \circ F^N(p)) - N = \zeta(\pi_1 \circ F^N(q)) - N$. Since $\zeta$ is injective in $V_{\epsilon}$ we have $F^N(p) = F^N(q)$ and therefore $p =q$.
\item $\Phi$ is surjective. By the definition of $\Phi$, we clearly have:
\begin{align*}
\Phi(\Gamma) = \bigcup_{n=0}^{\infty} U - n = \bigcup_{n=0}^{\infty}  \{\Re e \zeta > R-n\} =  \CC
\end{align*}
and we get $\Phi$ is surjective.
\end{itemize}
Therefore $\Gamma$ is biholomorphic to $\CC$.
\end{proof} 
Now we describe the open region $U$ attracted to the origin along the \textit{degenerate} characteristic direction $(0,1)$.

We will show later that the attracting curve $\mathcal{C}$, as in \eqref{invariantbigcurve} and this open region $U$ are disjoint. In fact, we will prove that the attracting curve is contained in the boundary of the basin of the open region.

The main proposition is the following:
\begin{proposition}\label{prop:degenerateopenregion}
Define $\displaystyle{D_{(z,w)} := \{(z,w) \in \CC^2 : z \in V_\epsilon, w-\gamma(z) \in V_\epsilon, |z|<|w-\gamma(z)|\}}$. 
Then $F(D_{(z,w)}) \subset D_{(z,w)}$.

Moreover, denote $F^n(z,w) = (z_n,w_n)$ for $(z,w) \in D_{(z,w)}$. Then:
\begin{eqnarray}
|z_n| \sim \frac{1}{n} 
\end{eqnarray}
and
\begin{eqnarray}
|w_n| \sim \frac{1}{\log n}
\end{eqnarray}
\end{proposition}

\begin{proof}
Let $(z,w) \in D_{(z,w)}$, i.e. $z \in V_{\e}, u=w-\gamma(z) \in V_{\e}$ and $|z| < |u|$ where we introduce the new variable $u:= w - \gamma(z)$.
We want to prove:
\begin{itemize}
\item [(i)] \underline{$z_1 \in V_{\e}$} 
\begin{align*}
z_1 &= z + z^2 + O(z^3,z^2w) \\
&= z + z^2 +O(z^3, z^2(u+\gamma(z))) \\
&= z + z^2 +O(z^3,z^2u)
\end{align*}
Proving that $z_1 \in V_\e$ is equivalent to proving that $1/z_1 \in U_R$ where 
\begin{align}\label{Udefin}U_R:= \{\zeta \in \CC: |\zeta| > R, |\Arg(\zeta)-\pi|<\pi/8 \}
\end{align} 
and $R=1/\e$. 
Rewriting the equation for $z_1$ we see:
\begin{align}\label{z1simple}
\nonumber\frac{1}{z_1} &= \frac{1}{z + z^2 +O(z^3,z^2u)}\\
\nonumber&= \frac{1}{z} \left(\frac{1}{1+z + O(z^2,zu)}\right)\\
\nonumber& = \frac{1}{z}(1-z+O(z^2,zu))\\
\frac{1}{z_1}& = \frac{1}{z} - 1 +O(z,u)
\end{align}
By decreasing the value of $\e$ as necessary (or equivalently, increasing the value of $R$) it is clear than $|O(z,u)|<1/8$. So $1/z \in U_R$ implies $1/z_1 \in U_R$ and therefore $z_1 \in V_\e$. 
Later we will need a more refined estimate for $1/z_1$.
\item [(ii)] \underline{$u_1 = w_1 - \gamma(z_1) \in V_{\e}$}.
This is the most delicate part of the proof. We recall the second coordinate of $F$ and rewrite it in a convenient way for our purposes:
\begin{align}
w_1 &= w - w^2z - \frac{z^4}{3} +O(z^5, z^3w, z^2w^2, zw^3)\nonumber \\
\label{wsimpler}w_1 &= w - w^2z - l(z) +zw\theta(z,w)
\end{align}
where $\theta(z,w) = \sum_{i+j\geq2}a_{ij}z^iw^j$. We would like to express $u_1 = w_1 - \gamma(z_1)$ in terms of $u$ and $z$.
We need therefore to estimate $\gamma(z_1)$. Substituting $u+\gamma(z)$ for $w$ in \eqref{wsimpler} we obtain:
\begin{align}\label{wintermsofu}
u_1 + \gamma(z_1) = (u+\gamma(z)) - (u+\gamma(z))^2z - l(z) + z(u+\gamma(z))\theta(z,u+\gamma(z))
\end{align}
Recall that the curve $\mathcal{C}$ is invariant. This implies that $F(z,\gamma(z)) \in \mathcal{C}$, since $z\in V_\e$. Let $F(z,\gamma(z)) = (z',\gamma(z'))\in \mathcal{C}$. Then $z' \in V_\e$ and using equation \eqref{wsimpler}:
\begin{align} 
\label{gammazprima}\gamma(z') = \gamma(z) - \gamma(z)^2z - l(z) + z\gamma(z)\theta(z, \gamma(z))
\end{align}
Noting that Equation \eqref{gammazprima} and Equation \eqref{wintermsofu} are very similar and subtracting one from the other, we get:
\begin{align*}
u_1 + \gamma(z_1) - \gamma(z') =& u -(u^2+ 2u\gamma(z))z + zu\theta(z,u+\gamma(z))+ \\
& z\gamma(z)[\theta(z,u+\gamma(z))-\theta(z,\gamma(z)]
\end{align*}
Then solving for $u_1$:
\begin{align}
\label{u1largo}u_1 =& u -u^2z - 2uz\gamma(z) + zu\theta(z,u+\gamma(z)) + \\
&\nonumber z\gamma(z)(\theta(z,u+\gamma(z))-\theta(z,\gamma(z)) + \gamma(z') - \gamma(z_1)
\end{align}
\textsl{Claim}: The following holds for $z\in V_\e$ and $u \in V_\e$:
\begin{itemize}
\item [$\ast$] $\gamma(z') - \gamma(z_1) = O(z^4u)$
\item [$\ast$] $\theta\left(z,u + \gamma(z)\right)-\theta(z,\gamma(z)) = O(uz)$
\item [$\ast$] $zu\theta\left(z,u+\gamma(z)\right) = O(z^3u,z^2u^2,zu^3)$
\end{itemize}

Assume the claim is proved. Equation \eqref{u1largo} yields:
\begin{align}
\label{u1simple}u_1 = u-u^2z + O(uz^3,z^2u^2,zu^3)
\end{align}
We use again the same idea as in (i). Proving that $u \in V_\e$ is equivalent to proving that $1/u \in U_R$, where $U_R$ is defined as in equation \eqref{Udefin}. Rewriting the Equation \eqref{u1simple} we get:
\begin{align*}
\frac{1}{u_1} &= \frac{1}{u}\left(\frac{1}{1-uz+O(z^3,z^2u,zu^2)}\right)\\
&= \frac{1}{u}(1+uz+O(z^3,z^2u,zu^2))\\
&= \frac{1}{u} + z + O(z^3/u, z^2, zu)
\end{align*}
Since we are assuming $|z| < |u|$ we have $|z^3/u| < |z^2|$. Shrinking $\e$ if necessary, we get $|O(z^3/u,z^2,zu)| < \frac{1}{8}|z|$. Since $z\in V_\e$, we see that $1/u \in U_R$ implies $1/u_1 \in U_R$, and hence $u_1 \in V_\e$.

We now prove the claim. 
\begin{itemize}
\item[$\ast$] $\gamma(z') - \gamma(z_1)$.
For this term we have to use the previous estimates for $\gamma$ and its derivative, as well as an estimate for $z' - z_1$:
\begin{align*}
z' - z_1 = F_1(z,\gamma(z)) - F_1(z, \gamma(z) + u) = O(z^2u)
\end{align*}
where $F_1$ is the first coordinate of $F$.
Note that both $z_1$ and $z'$ are in $V_\e$. Thus \eqref{gammasize} and \eqref{gammaprimesize} imply:
\begin{align*}
\gamma(z') - \gamma(z_1) &= \gamma'(z_1) (z' - z_1) + O((z'-z_1)^2)\\
		&= O(z^2)O(z^2u) + O(z^4u^2) = O(z^4u)
\end{align*}
\item[$\ast$] $\theta\left(z,u + \gamma(z)\right)-\theta(z,\gamma(z)) = O(uz)$.
For this term we use $\theta(z,w) =O(z^2,zw,w^2)$ and therefore
\begin{align*}
\theta\left(z, u + \gamma(z)\right) - \theta(z,\gamma(z))= O(zu,u^2)
\end{align*}
\item[$\ast$] $zu\theta\left(z,u+\gamma(z)\right) = O(z^3u,z^2u^2,zu^3)$: 
This follows from $\theta(z,u+\gamma(z)) = O(z^2,zu,u^2)$.
\end{itemize}
This completes the proof of the claim and of (ii).

\item [(iii)] \underline{$|z_1| < |u_1|$.}
This is equivalent to proving that $|1/z_1|>|1/u_1|$. From part (i) and (ii) we obtain $|1/z_1|>|1/z|+1/2$ and $|1/u_1| < |1/u| + |2z|$. Since $|z|<|u|$ we got $|1/z|>|1/u|$. By shrinking $\e$ we have $|1/z_1|>|1/z| + 1/2>|1/u|+2|z|>1/|u_1|$, and we get $|z_1|<|u_1|$.
\end{itemize}
Therefore, we proved $F(D_{(z,w)}) \subset D_{(z,w)}$.
If $(z,w) \in D_{(z,w)}$, then $(z_k,w_k):= F^k(z,w) \in D_{(z,w)}$ for all $k \geq 0$. Set $u_k = w_k - \gamma(z_k)$. Therefore we have:
\begin{align*}
\frac{1}{z_{k+1}} = \frac{1}{z_k} - 1 + O(z_k,u_k)
\end{align*}
summing from $k=0$ to $N$, and dividing by $N$, we get:
\begin{align*}
\frac{1}{Nz_N} = \frac{1}{Nz_0} - 1 + O(\sum_k{z_k,u_k})/N.
\end{align*}
Letting $N$ tend to infinity we get
\begin{align*}
\lim_{N \to \infty} \frac{1}{Nz_N} = -1
\end{align*}
In the same way for $u$:
\begin{align*}
\frac{1}{u_{k+1}} = \frac{1}{u_k} + z_k + O(z_k^3/u_k,z_k^2,z_ku_k),
\end{align*}
summing from $k=0$ to $N$, and dividing by $\log N$:
\begin{align*}
\frac{1}{\log N u_N} = - \frac{\sum_k z_k}{\log N} + O(\sum_k{z_k^3/u_k,z_k^2,z_ku_k})/ \log N
\end{align*}
and, letting $N$ tend to infinity,
\begin{align*}
\lim_{N \to \infty} \frac{1}{\log N u_N} = -1 
\end{align*}
which implies:
\begin{align*}
|z_N| \sim 1/N\\
|w_N| \sim 1/\log N.
\end{align*}
\end{proof}
Note that this does not prove that our curve $\mathcal{C}$ is in the boundary of the region $D_{(z,w)}$. We prove a weaker statement now:
\begin{proposition} \label{tododomain}
There exists $\delta>0$ and $N>0$, such that if $(z,w) \in \CC^2$, $z \in V_\delta$ and $w-\gamma(z) \in V_\delta$, then $F^N(z,w) \in D_{(z,w)}$.
\end{proposition}

\begin{proof}
Let $u = w -\gamma(z)$. We are considering any $z,u \in V_\delta$, and we want to prove that there exists $N$ large such that $|z_N| < |u_N|$ and $z_N, u_N \in V_\e$.
If $|z| < |u|$, then there is nothing to do. If $|z| > |u|$, going back to the proof of the last proposition we will still have $z_1 \in V_\e$, since we did not use $|z| < |u|$ in the proof of $z_1 \in V_\e$. The only part of the proof in which we use $|z|<|u|$ was to prove $u_1 \in V_\e$. Recall from equation \eqref{u1simple}:
\begin{align*}
u_1 = u - u^2z + O(z^3u,z^2u^2,zu^3).
\end{align*}
Since we do not have $|z|<|u|$, we can not replace $|z^3u| < |zu|^2$. Therefore we do not get $u_1 \in V_\e$. Nonetheless, we have the following estimate:
\begin{align*}
\frac{1}{u_1} = \frac{1}{u}(1+O(z^3)) + z + O\left(z^2,zu\right)
\end{align*}
and therefore
\begin{align*}
\frac{1}{u_N} = \frac{1}{u}\prod_i(1+O(z_i^3)) + \sum_i{z_i} + \sum_i O\left({z_i}^2,z_iu_i\right).
\end{align*}
We choose $\delta$ small enough, so we have $|u_i| < \e$, for all $1\le1 i \leq N$. Consequently we have $z_i \in V_e$ for all $1\leq i \leq N$, since $z_{i+1} \in V_\e$ depends only on the size of $u_i$. 
We will have:
$$
\frac{1}{z_N} \sim \frac{1}{z} - N 
$$ and 
$$\frac{1}{u_N} = \frac{1}{u}\prod_i(1+O(z_i^3)) + \sum_i{z_i} + \sum_i O\left({z_i}^2,z_iu_i\right).$$ 
Since $|z_k| = O(1/k)$, for every $k<N$, the term $\prod_i(1+O(z_i^3)) \sim \exp(\sum_i(O(z_i^3))< C|z_N|<2C/N$, where $C$ is a finite number. Therefore, the dominant term in the expression for $u$ is $\sum_i{z_i} \sim \log N$. This means that choosing $N$ large enough, or equivalently choosing $\delta$ small enough, we have $|z_N| \sim 1/N < 1/\log N \sim |u_N|$.  
\end{proof}

\begin{corollary}\label{partCbound}
There exists $\delta > 0$ and $N > 0$; such that 
\begin{align}\label{parte1}F^N(\mathcal{C} \cap B_\delta(O)) \subset\partial D_{(z,w)}
\end{align}
\end{corollary}
\begin{proof}
We have $\mathcal{C}\cap B_\delta(O) \subset \partial(\{(z,w), z \in V_\delta, w-\gamma(z) \in V_\delta\})$. Applying $F^N$ and using Proposition \ref{tododomain} we get \eqref{parte1}. 
\end{proof}

\begin{remark} Note that the curve $\mathcal{C}$ is tangent to the $z$-axis to order $3$.

\end{remark}
It will be more useful for our purposes to change coordinates in $D_{(z,w)}$. We use the following change of coordinates (we have already used it implicitly):
\begin{equation}\label{zwchangetoxy}
(x,y) := \phi (z,w) = \left(\frac{1}{z},\frac{1}{w - \gamma(z)}\right)
\end{equation}
Let $D_{(x,y)} = \phi(D_{(z,w)})$. Clearly:
\begin{align} 
D_{(x,y)} = \{(x,y) \in \CC^2: x, y \in U_R, |y| < |x| \}
\end{align}
where $U_R$ is defined as in \eqref{Udefin}.
We will work from now on with the map $\tilde{F} = \phi^{-1}\circ F\circ\phi$. As we said before, we need more precise equations for $\tilde{F}$. We give them in the following proposition:

\begin{proposition} Let $\tilde{F}(x,y) = (x_1,y_1)$, where $\tilde{F}$ is defined above and $(x,y) \in D_{(x,y)}$. In these coordinates we have:
\begin{align}
x_1 &= x  - 1 + g\left(\frac{1}{y}\right) + \frac{c}{x} +O\left(\frac{1}{x^2},\frac{1}{xy}\right)
\\
y_1 &= y + \frac{1}{x} + \frac{1}{x}h\left(\frac{1}{y}\right) + O\left(\frac{1}{x^2}\right) 
\end{align}
for all $(x,y) \in D_{(x,y)}$, where $g$ and $h$ analytic functions in $V_\epsilon$, $g(1/y) =O(1/y)$ and $h(1/y) = O(1/y)$.
\end{proposition}

\begin{proof}
It is clear from the Proposition \ref{tododomain} that the map $\tilde{F}$ is well-defined. Indeed $F(D_{(z,w)}) \subset D_{(z,w)}$ and because $\mathcal{C} = \{(z,\gamma(z)), z\in V_\e\} \cap D_{(z,w)} = \emptyset$ we can invert $w -\gamma(z)$ in $D_{(z,w)}$.
In the proof of the last proposition we saw equation \eqref{z1simple}:
\begin{align*}
\frac{1}{z_1} = \frac{1}{z} - 1 +O(z,u)
\end{align*}
We put together all the terms in $O(z,u)$ that contain only $u$ terms and we call this function $g(u)$. This function is analytic in a neighborhood of $0$ and in particular in $V_\e$. We also separate the linear term on $z$:
\begin{align*}
\frac{1}{z_1} = \frac{1}{z} - 1 + g(u) + cz + O(z^2,zu)
\end{align*}
where $c \in \CC$ is a constant. In the new coordinates $(x,y)$ we have:
\begin{align*}
x_1 = x - 1 + g\left(\frac{1}{y}\right) + \frac{c}{x} + O\left(\frac{1}{x^2},\frac{1}{xy}\right)
\end{align*}
Similarly for $y = 1/(w-\gamma(z)) = 1/u$ we have from equation \eqref{u1simple}
\begin{align*}
\frac{1}{u_1}= \frac{1}{u} + z + O\left(\frac{z^3}{u},z^2,zu\right)
\end{align*}
and rewriting it in terms of $x$ and $y$, we obtain:
\begin{align*}
y_1&= y + \frac{1}{x} + O\left(\frac{y}{x^3},\frac{1}{x^2},\frac{1}{xy}\right)\\
&=y + \frac{1}{x} + \frac{1}{x}h\left(\frac{1}{x}\right) + O\left(\frac{y}{x^3},\frac{1}{x^2}\right)
\end{align*}
Since we have $|y| < |x|$,  $O(y/x^3)$ is bounded by $O(1/x^2)$. Also, we separated the terms of the form $zu^k$, or equivalently, the terms $1/xy^k$. So we have:
\begin{align*}
y_1&= y + \frac{1}{x} + \frac{1}{x}h\left(\frac{1}{y}\right) + O\left(\frac{1}{x^2}\right)
\end{align*}
where $h$ is an analytic function defined for $y \in U_R$. The proposition is proved.
\end{proof}

Notice that $D_{(z,w)}$ is the invariant region for $F$ and $D_{(x,y)}$ is the invariant region for $\tilde{F}$. The following diagram commutes:
\[ 
\begin{CD} 
D_{(z,w)} @>F>> D_{(z,w)} \\ 
@V\phi VV  @V\phi VV   \\ 
D_{(x,y)} @>\tilde{F}>> D_{(x,y)} 
\end{CD} \] 


In the following section we introduce new coordinates on $D_{(z,w)}$ in which our map takes a simpler form.

\section{Semiconjugacy to translation}

In this section we change coordinates $(z,w) \to (\psi(z,w),w)$ in $D_{(z,w)}$ so we have (as in Weickert, or Hakim) that $\psi$ is an Abel-Fatou coordinate. More precisely, we want to find a map $\psi: D_{(z,w)} \to \CC$ such that
\begin{equation}
\label{fatoucoordzw}\psi \circ F (z,w) = \psi (z,w) - 1
\end{equation}
for all $(z,w) \in D_{(z,w)}$.

This coordinate is known in the literature as the Abel-Fatou coordinate. It always exists for maps tangent to the identity in one dimension; see Milnor \cite{mi}.

So far we have introduced one biholomorphic change of coordinates as in equation \eqref{zwchangetoxy} that transform $D_{(z,w)}$ as follows:
\[
(x, y) := \phi(z,w) = \left(\frac{1}{z}, \frac{1}{w - \gamma(z)}\right)
\]

Notice that this can be translated to our new coordinates $(x,y)$ in the following way: if $\mu$ is defined as $\mu := \psi \circ \phi^{-1}$ then we have:
$$
\mu: D_{(x,y)} \to \CC
$$
and
\begin{equation}\label{fatoucoordxy}
\mu \circ \tilde{F} (x,y) = \mu (x,y) -1
\end{equation}
for all $(x,y) \in D_{(x,y)}$.

So, finding a solution for \eqref{fatoucoordxy} and for \eqref{fatoucoordzw} are equivalent problems. We will first find $\mu$ that solves \eqref{fatoucoordxy} and then translate it to the original coordinates.

Let us recall how our map looks in the $(x,y)$ coordinates.
\begin{eqnarray*}
x_1 &=& x - 1 + g\left(\frac{1}{y}\right) + \frac{c}{x} + O\left(\frac{1}{x^2}, \frac{1}{xy}\right)\\
y_1 &=& y + \frac{1}{x} +\frac{1}{x}h\left(\frac{1}{y}\right) + O\left(\frac{1}{x^2}\right)
\end{eqnarray*}

So, our coordinate $x \to x_1$ is close to being a translation. Nonetheless we have several terms to deal with, such as $g(1/y)$ and $c/x$. We prove the following lemma, which will help us clean up some of these inconvenient terms.

\begin{lemma}
Choose and fix a branch of the logarithm in $U_R$. For $(x,y)$ in $D_{(x,y)} = U_R \times U_R$, with $R$ big enough we have:
\begin{itemize}
\item[a)]
\begin{equation}\label{logxchange}
\log(x_1) - \log(x) = -\frac{1}{x} + O\left(\frac{1}{x^2}, \frac{1}{xy}\right)
\end{equation}
\item[b)]
\begin{equation}\label{logychange}
\log(y_1) - \log(y) = \frac{1}{xy} + O\left(\frac{1}{x^2y},\frac{1}{xy^2}\right)
\end{equation}
\item[c)]
There exists a holomorphic solution $\beta: U_R \to \CC$ to the following first-order linear differential equation:
\begin{equation}\label{betachange}
\left(1+h\left(\frac{1}{y}\right)\right)\beta'(y) - \left(1 - g\left(\frac{1}{y}\right)\right)\beta(y) =  - g\left(\frac{1}{y}\right)
\end{equation}
such that $\beta(y) = O(1/y) $ and $\beta'(y) = O(1/y^2)$, where $x,y \in U_R$. If $\alpha(x,y) = x \beta(y)$ then we have the following estimate:
\begin{equation}\label{alphachange}
\alpha(x_1,y_1) = \alpha(x,y) - g\left(\frac{1}{y}\right)+ O\left(\frac{1}{x^2}, \frac{1}{xy}\right).
\end{equation}
\end{itemize}
\end{lemma}

\begin{proof}
We start by proving \eqref{logxchange}:
\begin{eqnarray*}
\log\left(\frac{x_1}{x}\right) &=& \log\left(\frac{x - 1 + O(1/x, 1/y)}{x}\right)\\
&=& \log\left(1 - \frac{1}{x} + O\left(\frac{1}{x^2}, \frac{1}{xy}\right)\right)\\
&=& - \frac{1}{x} + O\left(\frac{1}{x^2}, \frac{1}{xy}\right)
\end{eqnarray*}

Now, we prove Equation \eqref{logychange}:
\begin{eqnarray*}
\log\left(\frac{y_1}{y}\right) &=& \log\left(\frac{y + 1/x+ O(1/x^2, 1/xy)}{y}\right)\\
&=& \log\left(1 + \frac{1}{xy} + O\left(\frac{1}{x^2y}, \frac{1}{xy^2}\right)\right)\\
&=& \frac{1}{xy} + O\left(\frac{1}{x^2y}, \frac{1}{xy^2}\right)
\end{eqnarray*}

Finally we prove \eqref{alphachange}. This is probably the hardest part of the proof, so we give each step very carefully. 

First we want to compute $\alpha(x_1, y_1) - \alpha(x,y)$, where $\alpha$ is defined as in the proposition. 
\begin{align*}
\alpha(x_1, y_1) - \alpha(x,y) &= x_1\beta(y_1) - x\beta(y) \\
&= x_1\beta(y_1) - x_1\beta(y) + x_1\beta(y) - x\beta(y)\\
&= x_1[\beta(y_1) - \beta(y)] + \beta(y)[x_1 - x]
\end{align*}

Recall that $\displaystyle{y_1 - y = \frac{1}{x}[1+h\left(\frac{1}{y}\right)] + O(1/x^2)}$ and we will prove later that $\beta'(y) = O(1/y^2)$. We first compute $\beta(y_1) - \beta(y)$:
\begin{align*}
\beta(y_1) - \beta(y) =& (y_1 - y)\beta'(y) + O\left(\beta''(y)|y_1 - y|^2\right)\\
=& \beta'(y) \left[\frac{1}{x}\left(1+h\left(\frac{1}{y}\right)\right) + O\left(\frac{1}{x^2}\right)\right] + O\left(\frac{1}{x^2}\right)\\
=& \frac{1}{x}\left(1+h\left(\frac{1}{y}\right)\right)\beta'(y) + O\left(\frac{1}{x^2}\right) 
\end{align*}
so, putting it back we get:
\begin{align*}
\alpha(x_1, y_1) - \alpha(x,y) =& x_1\left[\frac{1}{x}\left(1+h\left(\frac{1}{y}\right)\right)\beta'(y) + O\left(\frac{1}{x^2y^2}\right)\right] \\
&+\beta(y)\left[-1+g\left(\frac{1}{y}\right) + O\left(\frac{1}{x}\right)\right]\\
=& (x+O(1))\left[\frac{1}{x}\left(1+h\left(\frac{1}{y}\right)\right)\beta'(y)\right]\\
& - \left(1 - g\left(\frac{1}{y}\right)\right) \beta(y) + O\left(\frac{1}{xy}\right),
\end{align*}
since  $\beta(y) = O\left(1/y\right)$. So, we see:
\begin{eqnarray*}
\alpha(x_1,y_1) - \alpha(x,y) = \left(1+h\left(\frac{1}{y}\right)\right)\beta'(y) - \left(1- g\left(\frac{1}{y}\right)\right)\beta(y) +  O\left(\frac{1}{x^2},\frac{1}{xy}\right)
\end{eqnarray*}
From \eqref{betachange} we then get:
\begin{eqnarray*}
\alpha(x_1,y_1) - \alpha(x,y) = -g\left(\frac{1}{y}\right) +  O\left(\frac{1}{x^2},\frac{1}{xy}\right)
\end{eqnarray*} and Equation \eqref{alphachange} is proved.
It remains to prove is the estimates on the solution of the differential equation \eqref{betachange}.

Recall that $g$ and $h$ are holomorphic equations defined in a neighborhood of the origin and $g(z) = O(z)$ and also $h(z) = O(z)$. Choosing $R$ large enough we can guarantee that $|h(1/y)| < 1/4$ for all $y \in U_\R$, and therefore we can divide by $1+h(1/y)$ to obtain:
\begin{align*}
\beta'(y) - \frac{1 - g(1/y)}{1+h(1/y)}\beta(y) =  - \frac{g(1/y)}{1+h(1/y)},
\end{align*}
renaming for convenience we have the following differential equation
\begin{align*}
\beta'(y) - \left(1 + s\left(y\right)\right)\beta(y) =  t\left(y\right),
\end{align*}
where $s(z) = O(1/z)$ and similarly $t(z) = O(1/z)$.

If we call $\sigma(y) = \beta(y) + t(y)$, the differential equation will become
\begin{align*}
\sigma'(y) - \left(1 + s\left(y\right)\right)\sigma(y) =  \tau\left(y\right)
\end{align*}
where $s(z) = O(1/z)$ and similarly $\tau(z) = O(1/z^2)$. We will prove that the solution of the equation $\sigma(y) = O(1/y)$, and therefore we will have $\beta(y) = O(1/y)$.

By increasing $R$ if necessary we have all the estimates for $s, t$ and $\tau$ in an open neighborhood of $\overline{U_R}$. Using differential equations theory, we see that one solution of this differential equation is
\begin{align*}
\sigma(y) = \frac{\int_{-R}^y H(\zeta)\tau(\zeta) d\zeta}{H(y)},
\end{align*}
where the integral is in any path joining the points $-R$ and $y$ in $U_R$ and:
$$
\frac{H'(y)}{H(y)} = - 1  - s(y).
$$
Notice that all the integrals are independent of the path chosen because we are working in $U_R$, which is a simply connected region. We prove the following
$$
\max_{\zeta \in \gamma}|H(\zeta)\tau(\zeta)| = |H(y)\tau(y)|
$$

If $\tau \equiv 0$ there is nothing to prove. Assume then $\tau(y) = \sum_{n \geq k} a_n/y^n$, where $k$ is the lowest index such that $a_k \neq 0$. Then
$
\tau(y) = \frac{a_k}{y^k}(1+ O(1/y))
$
by increasing $R$ we can make $|O(1/y)| < 1/2$. So we have $\tau(y) \neq 0$ for all $y \in U_R$. We can choose a branch of logarithm for $\tau(y)$. Let $M(y) = \log (\tau(y))$ or equivalently $\tau(y) = \exp(M(y))$. We have
\begin{align*}
M(y) &= \log \left(\frac{a_k}{y^k} \left(1+O\left(\frac{1}{y}\right)\right)\right)
= \log \left(\frac{a_k}{y^k}\right) + \log\left(1+ O\left(\frac{1}{y}\right)\right)\\
&= \log a_k - k\log y + O\left(\frac{1}{y}\right).
\end{align*}
From the definition of $H$, we have:
$$\log (H(y)) = L(y) = - y  - b_1 \ln y - \sum_{i \geq 2}\frac{b_i}{(i-1)y^{i-1}},$$ where $L'(y) = -1 - s(y)$ and $s(y) =  \sum_{i \geq 1}\frac{b_i}{y^i}$. 

We call $u(t) = |\exp(M(\zeta(t))+ L(\zeta(t))|$, where $\zeta(t) = (1-t)(-R) + ty$ is a parametrization of the line that joins the points $-R$ and $y$. We rewrite what we want to prove as $\displaystyle{\max_{0 \leq t\leq 1}u(t) = u(1)}.$
We are gonna prove $u'(t) \geq 0$ for $0 \leq t \leq 1$, if we choose $y$ large enough.

We have
\begin{align*}
u(t) = |\exp(M(\zeta(t)) + L(\zeta(t))| =\exp \Re \left[M(\zeta(t)) + L(\zeta(t))\right] 
\end{align*}
Therefore
$u'(t) = u(t) \Re [(M(\zeta) + L(\zeta))'|_{\zeta = \zeta(t)}(y+R)]$
and
\begin{align*}
[M(\zeta)+ L(\zeta)]' &= M'(\zeta) + L'(\zeta)\\
&= [-\frac{k}{\zeta} + O\left(\frac{1}{\zeta^2}\right)]+ [-1 + O(1/\zeta)]\\
&= -1 + O(1/\zeta)
\end{align*}
Then $u'(t) = u(t)\Re[(y+R)(-1 + O(1/\zeta(t))]$
By considering $y$ large enough we can guarantee $\Re[(y+R)(-1+O(\frac{1}{R+tR+ty}))]$ is positive.
So, we have $u'(t) \geq 0$, then $\max_{\zeta \in \gamma}|H(\zeta)\tau(\zeta)| = |H(y)\tau(y)|$. Putting back on the solution $\sigma$ we will have $|\sigma(y) | 
\leq |\tau(y)| O(y)= O(1/y)$. Therefore $\beta(y) = O(1/y)$.

\end{proof}

We use these functions to define the Abel-Fatou coordinate. In the $(x,y)$ coordinates, the first coordinate map looks as follows:
\begin{eqnarray*}
x_1 &=& x - 1 +g\left(\frac{1}{y}\right) + \frac{c}{x} + O\left(\frac{1}{x^2}, \frac{1}{xy}\right).
\end{eqnarray*}
By using the last lemma and the auxiliary functions we found we will be able to prove the existence of $\mu$ such that $\mu(x_1,y_1) = \mu(x,y) - 1$. 

\begin{theorem}\label{changemu}Choose and fix a branch of the logarithm on $U_R$.
Define:
$$
\mu_n(x,y) = x_n + n + r\log(x_n) +  s\log(y_n) + \alpha(x_n,y_n),
$$
with $r$ and $s$ constants well chosen. Then $\mu_n(x,y)$ is Cauchy  with the uniform norm topology and therefore we define the limit function as $\mu(x,y)$.
\end{theorem}

\begin{proof}
We want to show that $\mu_n(x,y)$ is Cauchy and therefore converges uniformly in $U_R \times U_R$.
Using the lemma above, we have:
\begin{align*}
\mu_{n+1} - \mu_n  = &\left( x_{n+1} + n+1 + r\log x_{n+1} +  s\log y_{n+1} + \alpha(x_{n+1},y_{n+1})\right) - \\
& - (x_n + n + r\log x_n +  s\log y_n + \alpha(x_n,y_n)) \\
= &(x_{n+1} - x_n) + 1 + r(\log x_{n+1} - \log x_n) + \\
& + s(\log y_{n+1} - \log y_n) + (\alpha(x_{n+1},y_{n+1}) - \alpha(x_{n},y_{n})) \\
= & \left( - 1 + g\left(\frac{1}{y_n}\right) + 1 + \frac{c}{x_n} \right) - \frac{r}{x_n} +\\
& + s\left( \frac{1}{x_ny_n}\right) - g\left(\frac{1}{y_n}\right) + \frac{k}{x_ny_n} + O\left(\frac{1}{x_n^2},\frac{1}{x_ny_n^2}\right)
\end{align*}
We choose $r = c$ and $s = -k$, where the term $k/x_ny_n$ is the sum of all the terms of this form in the other factors. Therefore:
\begin{align*}
|\mu_{n+1} - \mu_n| = \left|O\left(\frac{1}{x_n^2}, \frac{1}{x_ny^2_n}\right)\right|
\end{align*}
Using the fact that $|x_n| \sim n$ and $|y_n| \sim \log n$ we can see that these terms add up and converge. 
Therefore:
\begin{equation}
\mu_n - \mu_0 = \sum_{i=0}^{n}\left(\mu_{i+1} - \mu_i\right)
\end{equation}
converges absolutely uniformly on $D_{(x,y)}$ to a holomorphic limit $\mu$. 
Let $\mu = \lim \mu_n = \mu_0 + \eta$. We have:
\begin{align*}
|\eta| = |\sum_n (\mu_{n+1} - \mu_n) | \leq \sum_n \left|O\left(\frac{1}{n^2}, \frac{1}{n\log^2n}\right)\right| = O\left(\frac{1}{n},\frac{1}{\log n}\right)
\end{align*}
Our estimates above show that $\eta \to 0$ uniformly in $D_{(x,y)}$ as $(x,y) \to \infty$.
\end{proof}

We summarize what we have found in the following proposition:
\begin{proposition}
There exists a map $\psi : D_{(z,w)} \to \CC$ such that:
\begin{equation*}
\psi(F(z,w)) = \psi(z,w) - 1
\end{equation*}
for all $(z,w) \in D_{(z,w)}$.
We have:
\begin{align}\label{zwchangemu}
\nonumber\psi(z,w) &= \mu(x,y) = \mu_0(x,y) + \eta(x,y) \\
&=x  + r \log(x) + s\log(y) +\alpha(x,y) + \eta(x,y)
\end{align}
\end{proposition}

\begin{proof}
Let $\displaystyle{\mu = \lim \mu_n = \mu_0 + \eta = x + r\log(x) + s\log(y) +\alpha(x,y) + \eta}$; and $\psi(z,w) = \mu(x,y) = \mu(\phi(z,w))$. Then:
\begin{eqnarray*}
&&\psi \circ F(z,w) = \lim_{n \to \infty} \left[ x_{n+1} + n + r \log x_{n+1} + s\log y_{n+1} + \alpha(x_{n+1},y_{n+1}) + \eta(x_{n+1},y_{n+1}) \right]\\
& = &  \lim_{n \to \infty} \left[ x_{n+1} + n + 1 + r\log x_{n+1} + s\log y_{n+1} + \alpha(x_{n+1},y_{n+1}) + \eta(x_{n+1},y_{n+1}) \right] -  1\\
& = & \psi(z,w) - 1,
\end{eqnarray*}
and so \eqref{fatoucoordzw} is satisfied.
\end{proof}

Consider the mapping from $D_{(z,w)}$ to $\CC^2$ given by:
$$
\Theta(z,w) = \left(\psi(z,w), \frac{1}{w-\gamma(z)}\right) = : (t,y)
$$
In the next proposition we prove $\Theta$ is a change of coordinates.

\begin{proposition}
$\Theta = (t,y)$ is a biholomorphism from $D_{(z,w)}$ onto its image $D_{(t,y)}$.
\end{proposition}

\begin{proof}
We have to show that $(t,y)$ is injective in $D_{(z,w)}$, or equivalently in $D_{(x,y)}$. If
$$
(\mu(x_1,y_1), y_1) = (\mu(x_2, y_2), y_2)
$$
for $(x_1,y_1)$ and $(x_2,y_2) \in D_{(x,y)}$ then we have:
$$
y_1 = y_2
$$
and
$$
\mu(x_1, y) = \mu(x_2,y).
$$
Recalling how we define $\mu(x,y)$ then we should have:
\begin{eqnarray*}
x_1  + r\log x_1 + s\log y + \alpha(x_1,y) + \eta(x_1,y) = x_2 + r\log x_2 + s\log y + \alpha(x_2,y) + \eta(x_2,y)\\
x_1 - x_2 + r\left(\log x_1 - \log x_2 \right)  + \alpha(x_1,y) - \alpha(x_2,y) + \eta(x_1,y) - \eta(x_2,y) = 0
\end{eqnarray*}
In case $x_1 \neq x_2$ we can divide by $x_1 - x_2$:
\begin{eqnarray}
\label{injective}1 + r\frac{\log x_1  - \log x_2 }{ x_1 - x_2}  + \frac{\alpha(x_1,y) - \alpha(x_2,y)}{x_1 - x_2}  + \frac{\eta(x_1,y) - \eta(x_2,y)}{x_1 - x_2} = 0
\end{eqnarray}
We have that $\eta(x,y)$ is bounded, so by increasing $R$ we can assume $\displaystyle{\frac{\eta(x_1,y) - \eta(x_2,y)}{x_1 - x_2}}$ is less than $\epsilon$. Also:
$$
\left|\frac{\log(x_1) - \log(x_2)}{x_1 - x_2}\right| = \left|\frac{1}{\tilde{x}}\right|
$$
for some $\tilde{x} \in V_R$ therefore: $\displaystyle{\left|r\frac{\log(x_1) - \log(x_2)}{x_1 - x_2}\right| <  \frac{|r|}{R}}$.

And for the third term:
$$\displaystyle{\frac{\alpha(x_1,y) - \alpha(x_2,y)}{x_1 - x_2} = \beta(y)  = O\left(\frac{1}{y}\right)},$$ we have $|\beta(y)| < C/|y| < C/R$.

Therefore we can choose $R$ large enough such that:
\begin{eqnarray*}
1 > \frac{|r|}{R} + \frac{C}{R}+ \epsilon
\end{eqnarray*}

So our map is injective and we have a valid change of coordinates.
\end{proof}

In the following section we modify the open region $D_{(t,y)}$ for convenience, so it will be easier to visualize each fiber of $\Psi$.

\section{Modification of $D_{(z,w)}$}

In this section we investigate the geometry of the image $D_{(t,y)}$ of $D_{(z,w)}$ in the new coordinates $\Theta=(t,y)$.  

We want to prove:
\begin{lemma}
$D_{(t,y)} \supset D'_{(t,y)}$ where 
$$
D'_{(t,y)} = \{(t,y) \in U_{2R} \times U_{R} : |y| < \frac{|t|}{2}\}.
$$
\end{lemma}

\begin{proof}
Fix $y_o \in U_{R}$, we call $\mu_{y_o}(x): = \mu(x,y_o)$. Define:
\begin{align*}
D_{y_o} = \{x \in \CC: (x,y_o) \in D_{(x,y)}\}
\end{align*}
By the definition of $D_{(x,y)}$ we have $D_{y_o}(x) = \{x \in U_R, |y_o|<|x|\}$. We want to prove $\mu_{y_o}(D_{y_o}) \supset U_{2|y_o|}$. We have
$\mu_{y_o}(x) = x +  r\log(x)+ s \log(y_o) + \alpha(x,y_o) + \eta(x,y_o)$.
Notice that $\log(y) = \log(|y_o|e^{i\theta})= \log|y_o| + i\theta$ where $\theta \in \left[\pi - \pi/8, \pi + \pi/8\right]$, which is basically just a translation along the real axis, since $\log|y_o| \gg \pi$ and since $x \in D_{y_o}$ we have $\log|y_o| < \log|x|$. Also $\eta(x,y_o)$ is bounded, therefore we can assume $|\eta(x,y_o)| < 1$.  The only term we need to estimate is $\alpha(x,y_o)$. 
We can choose $|\alpha(x,y)| < |x|/10$, since $\alpha(x,y) = x\beta(y)$ and $\beta(y) = O(1/y^2)$. Putting all together we have:
\begin{eqnarray*}
|\mu_{y_o}(x) - (x+(r+s)\log(x))| < \frac{|x|}{10} +1
\end{eqnarray*}
We will have that the image of $D_{y_o} = \{x \in U_R: |y_o|<|x|\}$ by $\mu$ will not change the asymptotic behavior of $U_R$ and  $\mu_{y_o}(D_{y_o}) \supset U_{2|y_o|}$.
This proves that $D_{(t,y)}$ contains $D'_{(t,y)} = \{(t,y) \in U_{2R} \times U_{R} : |t| > 2|y| \}$.
\end{proof}

It follows from the asymptotic behavior of $F$ in the coordinates $(t,y)$ that the region $D_{(t,y)}$ will be mapped into the region $D'_{(t,y)}$ under a finite number of iterates of F. So, we have $F^{k}(D_{(z,w)}) \subset D'_{(z,w)}$ for some $k$ big enough. Therefore, if we define $\Omega$ as follows:
\begin{equation}\label{defomega}
\Omega = \bigcup_{n\geq0}F^{-n}(D'_{(z,w)})
\end{equation}
then we also have:
\begin{equation*}
\Omega = \bigcup_{n\geq0}F^{-n}(D_{(z,w)})
\end{equation*}

From now on we work in $D'_{(z,w)}$, $D'_{(x,y)}$ and $D'_{(t,y)}$, where:
\begin{eqnarray*}
D'_{(t,y)} &=& (t,y)(D'_{(z,w)}) = \Theta(D'_{(z,w)})\\
D'_{(x,y)} &=& (x,y)(D'_{(z,w)}) = \phi(D'_{(z,w)}) \\
D'_{(t,y)} &=& \{(t,y) \in U_{2R} \times U_{R} : |y| < |t|/2 \} 
\end{eqnarray*}

All the open regions are displayed in the following diagram, where each vertical arrow is a biholomorphic change of coordinates and $\tilde{F}$ and $\tilde{\tilde{F}}$ is defined so the diagram commutes.

\[ 
\begin{CD} 
D'_{(z,w)} @>\subset>> D_{(z,w)} @>F>> D_{(z,w)} \\ 
@V\phi VV  @V\phi VV  @V\phi VV   \\ 
D'_{(x,y)} @>\subset>> D_{(x,y)} @>\tilde{F}>> D_{(x,y)} \\
@V(\mu,id) VV  @V(\mu,id) VV @V(\mu,id) VV\\ 
D'_{(t,y)} @>\subset>> D_{(t,y)} @> \tilde{\tilde{F}}>> D_{(t,y)} \\
\end{CD} \] 

We are now ready to prove that the invariant curve attracted to the origin along the nondegenerate characteristic direction is contained in the boundary of the invariant region attracted to the origin along the degenerate characteristic direction.

\begin{proposition}\label{gammaintheboundary}
For $\Gamma$ defined as $\Gamma = \cup F^{-n}(\mathcal{C})$, we have:
\begin{equation}
\Gamma \subset \partial \Omega
\end{equation}
\end{proposition}

\begin{proof}
First we claim that for $K$ large, $F^K(\mathcal{C}) \subset \partial D_{(z,w)}$. Assume the claim is proved. By definition $\Gamma = \bigcup_{n\geq0}F^{-n}(\mathcal{C})$ and $\Omega = \bigcup_{n\geq0}F^{-n}(D_{(z,w)})$.
Let $p \in \Gamma$, then $p \notin \Omega$, since iterating forward a finite number of times we should be in $\mathcal{C} \cap D_{(z,w)}$ which is empty. If $p \notin \partial \Omega$, then there exits a small ball around $p$, which is entirely outside $\Omega$. By iterating forward with $F$ a finite number of times, we should have a ball around $F^{k}(p) \in F^K(\mathcal{C}) \in \partial D_{(z,w)}$, but by shrinking enough, the whole ball should be outside $D_{(z,w)}$, which is a contradiction. Therefore $p \in \partial \Omega$. So $\Gamma \subset \partial \Omega$. Now, we prove the claim.
We have seen that $F(\mathcal{C}) \subset \mathcal{C}$. Let $\delta > 0$ and $N>0$ be as in corollary \ref{partCbound}. We proved $F^N(\mathcal{C} \cap B_\delta(0))\subset \partial D_{(z,w)}$. Since $\mathcal{C}$ is invariant, attracted to the origin by $F$, there exists $M$ large enough so $F^M(\mathcal{C}) \subset \mathcal{C} \cap B_\delta(0)$. So, $F^{M+N}(\mathcal{C}) \subset \partial D_{(z,w)}$ and the claim is proved.
\end{proof}

For $\Omega$ as in \eqref{defomega}:
\begin{equation*}
\Omega = \bigcup_{n\geq1}F^{-n}(D'_{(z,w)})
\end{equation*}
we can extend our coordinate $\psi = \mu\circ\phi$ to all of $\Omega$ in the usual way: 
Given $(z,w) \in \Omega$, there exists some $N$ such that $(z_N, w_N):= F^N(z,w) \in D'_{(z,w)}$. We define:
\begin{eqnarray*}
\psi(z,w) := \psi(z_N, w_N) + N
\end{eqnarray*}
It is straightforward to show that $\psi$ is well defined. From the equation it follows that $\psi(\Omega)$ covers $\mu(D'_{(x,y)}) + N$ for all $N$. Since $\psi(D'_{(z,w)})= \mu(D'_{(x,y)}) = U_{2R}$, for $R$ fixed, then we have:
\begin{equation}\label{firstcoordsurjective}
\psi(\Omega) = \CC
\end{equation}

Now we want to choose new coordinates in each fiber of $\psi$ and prove that $\Omega$ is biholomorphic to $\CC^2$.

\section{New coordinates on the fibers of $\psi$}

We start by proving that each fiber is connected and simply connected.

\begin{proposition}
For each $t \in \CC$, $\psi^{-1}(t)$ is connected and simply connected. 
\end{proposition}
\begin{proof}
For each $t \in \CC$ we can exhaust $\psi^{-1}(t)$ by the sequence:
$$
\psi^{-1}(t) = \bigcup_{n\geq0} W_n
$$
where
$$
W_n := \psi^{-1}(t) \cap F^{-n}(D'_{(z,w)}), 
$$
We have the following biholomorphism:
$$
F^{n}: \psi^{-1}(t) \cap F^{-n}(D'_{(z,w)}) \to \psi^{-1}(t-n) \cap D'_{(z,w)}
$$
From the last section we have that $\Theta$ biholomorphically maps $(z,w) \to (\psi(z,w), y)$.
Therefore for $n$ large enough, $\Theta=(t,y)$ maps $\psi^{-1}(t-n) \cap D'_{(z,w)}$ biholomorphically to $\{t-n\} \times T_{R, 2|t-n|}$ where we introduce the new notation
\begin{align}\label{tdefin}
T_{b,a} := \{\zeta \in \CC : \zeta \in U_b, |\zeta| < a\}
\end{align}
for $b<a$. Clearly, for $R$ large enough, $T_{R,2|t-n|}$ is connected and simply connected.
\end{proof}

We would like to prove now that each fiber is biholomorphic to $\CC$. As before, defining a function in $D'_{(z,w)}$ is equivalent to defining a function on $D'_{(t,y)}$. We will define a function $\xi$ on $D'_{(t,y)}$ and therefore, by composing with $\Theta$ we will get a function $\Upsilon$ in $D'_{(z,w)}$.

In the commutative diagram before we know the form for $F$, for $\tilde{F}$ but we have not computed the exact form for $\tilde{\tilde{F}}$. We will not need to compute it explicitly but have an estimate so we can map each fiber to $\CC$. 

We have:
\begin{align*}
x_1 &= x - 1 + g\left(\frac{1}{y}\right) + \frac{c}{x} + O\left(\frac{1}{x^2}, \frac{1}{x^2y}, \frac{1}{xy^2}\right)\\
y_1 &= y + \frac{1}{x} + O\left(\frac{1}{x^2},\frac{1}{x^2y}\right)
\end{align*}

and we have:
\begin{eqnarray*}
t_1 &=& t - 1\\
y_1 &=& G(x,y) = H(t,y) 
\end{eqnarray*}

We will not find $H$ exactly, but instead we will see it in terms of $t$ and $x$. For that purpose we use:
$$t = x + r\log x + s\log y + x\beta(y) + \eta(x,y).$$
Now:
\begin{eqnarray*}
\frac{1}{t} - \frac{1}{x} &=& \frac{1}{x + r\log x + s\log y + x\beta(y) + \eta} - \frac{1}{x}\\
&=& \left(-r\log x - s\log y - x\beta(y) - \eta \right)\frac{1}{x^2\left[1 + r\log x/x + s\log y/x + \beta(y) + \eta/x\right]}\\
&=& \left(-r \frac{\log x}{x^2} - s\frac{\log y}{x^2} - \frac{\beta(y)}{x} - \frac{\eta}{x^2} \right)\frac{1}{1 + r\log x/x+ s\log y/x + \beta(y) + \eta/x}\\
&=& \left(\frac{-r\log x}{x^2} - \frac{s\log y}{x^2} - \frac{\beta(y)}{x} - \frac{\eta}{x^2}\right)\left[1 + O\left(\frac{\log x}{x}, \frac{\log y}{x}, \beta(y), \frac{\eta}{x}\right)\right]\\
&=& - \frac{\beta(y)}{x} + O\left(\frac{\beta(y)^s2}{x},\frac{\log x}{x^2}\right)\\ &=& -\frac{1}{xy} + O\left(\frac{1}{xy^2},\frac{1}{x^{4/3}}\right).
\end{eqnarray*}
So, we have:
\begin{eqnarray*}
- \frac{1}{t} + \frac{1}{x} =  \frac{1}{xy} + O\left(\frac{1}{xy^2},\frac{1}{x^{4/3}}\right),
\end{eqnarray*}
where we write these terms since they are the largest ones asymptotically.
The term that we need to cancel is therefore $1/xy$. We can use the equation \eqref{logychange}, which tells us:
$$
\log(y_{n+1}) - \log(y_n) = \frac{1}{x_ny_n} + O\left(\frac{1}{x_n^2y_n},\frac{1}{x_ny_n^2}\right)
$$
to get rid of $1/xy$.
\begin{proposition}
Define:
$$
\xi_n(t,y) = y_n - \log(y_n) + \log(t_n)
$$
Then $\xi_n$ is Cauchy with the uniform norm topology. 
\end{proposition}

\begin{proof}
Using the relationships worked out above, we compute:
\begin{eqnarray*}
\xi_{n+1} - \xi_n &=& \left[y_{n+1} - \log(y_{n+1}) + \log(t_{n+1})\right] -  \left[y_n - \log(y_n) + \log(t_n)\right]\\
&=& y_{n+1} - y_n - \left[\log(y_{n+1}) - \log(y_n)\right] + \left[\log(t_{n+1}) - \log(t_n)\right]\\
&=& \frac{1}{x_n} + O\left(\frac{1}{x_n^2}\right) - \frac{1}{x_ny_n} + O\left(\frac{1}{x_n^2y_n},\frac{1}{x_ny_n^2}\right)
 - \frac{1}{t_n} + O\left(\frac{1}{t^2_n}\right)\\
&=&O\left(\frac{1}{x_ny_n^2},\frac{1}{x_n^{4/3}}, \frac{1}{x_n^2}, \frac{1}{t_n^2}, \frac{1}{x_n^2y_n}\right).
\end{eqnarray*}
Using $|x_n| \sim n$ , $|t_n| \sim n$ and $|y_n| \sim \log n$, we have:
\begin{eqnarray*}
\left|\xi_{n+1} - \xi_n\right| = O\left(\frac{1}{n^2\log n}, \frac{1}{n^{4/3}}, \frac{1}{n\log^2 n}\right) = O\left(\frac{1}{n\log^2n}\right)
\end{eqnarray*}
which is summable.
\end{proof}

Therefore we can define:
\begin{equation*}
\xi(t,y) := \lim_{n \to \infty} \xi_n(t,y).= \xi_0(t,y) + \eta'(t,y),
\end{equation*}
where $\eta'(t,y)$ is a bounded function.
Summarizing we get:
\begin{equation}\label{Upsilondefinition}
\Upsilon(z,w) := \xi(t,y) = y - \log(y) + \log(t) + \eta'(t, y).
\end{equation}

Notice here that
\begin{align*}
\left|\xi(t,y) - y\right| = \left| - \log(y) + \log(t) + \eta'(t,y)\right| < 2\log|y| + \log|t|
\end{align*}
which means that $\Upsilon(z,w)$, or equivalently $\xi(t,y)$, is still close to $y$, when $t$ is fixed. We will be able to extend this function $\xi$ to a larger domain in $D'_{(z,w)}$. In the next proposition we give an estimate for the image of $D'_{(t,y)}$ under the map $\xi$. Increase $R$ if necessary so we have $2R- 2\log 2R > R + 2\log R$. 

\pagebreak[2]
\begin{lemma}\label{Upsestimate}
Given the function $\xi$ defined as in equation \eqref{Upsilondefinition}, we have:
\begin{enumerate}
\item Fix $t_o \in U_{4R}$, we call $D^{t_o}= \{y \in \CC; (t_o,y) \in D'_{(t,y)}\}$.  Then $\xi_{t_o}(y)$ is injective in $D^{t_o}$ and $\xi(\{t_o\}\times D^{t_o}) \supset  T'_{R+2\log R, \frac{|t_o|}{2} - 2\log\frac{|t_o|}{2}} + \log(t_o)$. Where $T'_{a,b} = \{z \in \CC; a<|z|<b, |\Arg(z) - \pi|<\pi/20\}$ and $U + w = \{z+w: z\in U \}$.
\item For $(z,w) \in D_{(z,w)}$, we have:
$$
\Upsilon(z_1,w_1) = \Upsilon(z,w)
$$
\end{enumerate}
\end{lemma}

\begin{proof}
It is immediate that $\xi_{t_o}$ is injective, by using the relationship
$$
|\xi_{t_o}(y) - \log(t_o) - y| < 2\log|y|.
$$ 
Recall that $D'_{(t,y)} = \{(t,y) \in U_{2R}\times U_R; |y| < |t|/2 \}$. Fix $y_o \in U_{4R}$, the fiber is $D^{t_o} =\{y \in U_R, |y| < |t_o|/2\}$. Recall the definition of $T_{a,b}$ from equation \eqref{tdefin}, we have then $D^{t_o} = T_{R,|t_o|/2}$. Since $|\xi_{t_o}(y) - \log(t_o) - y| < 2\log|y|$ and for $y \in D^{t_o}$ we have $R<|y|<|t_o|/2$. Choose $R$ large enough so we have $2R\sin(\pi/8) > 4\log R$, which implies $2S\sin(\pi/8) > 4\log S$, for any $S >R$. Fix, $S$ any number $R<S<|t_o|/2$. We have $y \in D^{t_o}, |y| =S$ if and only if $y=Se^{i\pi\theta}$, where $|\theta - \pi| < \pi/8$. Then $|\xi_{t_o}(y)-\log(t_o)- y| < 2\log S$ i.e. $\xi_{t_o}(y) - \log(t_o) = Se^{i\pi\theta} + 2\log S e^{i\alpha}$, where $|\theta - \pi|<\pi/8$ and $\alpha$ is any number. Using the relation $2S\sin(\pi/8) > 4\log S$ we get $\Arg(\xi_{t_o}(y)) > \pi/20$. Also, $|\xi_{t_o}(y) - \log(t_o)|<|y| + 2\log|y| = S+2\log S $ and $|\xi_{t_o}(y)|>S -2\log S$. Run $S$ from $R$ to $t_o/2$ and we will get that $\xi_{t_o}(D^{t_o}) \supset T'_{R+2\log R,\frac{|t_o|}{2} - 2\log \frac{|t_o|}{2}} + \log(t_o)$ and the first claim is proved.
  
Now let us look at the relationship between $\Upsilon(z_1,w_1)$ and $\Upsilon(z,w)$: 
\begin{eqnarray*}
\Upsilon(z_1, w_1) &=&\xi(t_1, y_1) \\
&=& \lim_{n \to \infty} \left(y_{n+1} - \log(y_{n+1}) + \log(t_{n+1})\right)\\
&=& \lim_{n \to \infty} \left(y_{n} - \log(y_{n}) +\log(t_n)\right)\\
&=& \xi(t,y)\\
&=& \Upsilon(z,w)
\end{eqnarray*}
where $t = \psi(z,w)$. This completes the proof of the second claim.
\end{proof}

We extend $\Upsilon$ to $\{ p \in \Omega : t = \psi(p) \in U_{4R}\}$ by defining:
\begin{eqnarray}\label{Upsrecurse}
\Upsilon(p) = \Upsilon(F^n(p))
\end{eqnarray}
where $n$ is chosen so that $F^n(p) \in D'_{(z,w)}$. It is clear that this 
definition is independent of $n$ and that, for $t \in U_{4R}$, the mapping from $\psi^{-1}(t)$ to $\CC$ defined by $\Upsilon_t(p) = \Upsilon(t(p),y(p))$ is injective.

\begin{lemma}
The mapping $\Upsilon_t$ defined above is a biholomorphism of $\psi^{-1}(t)$ onto $\CC$, for every $t \in  U_{4R}$.
\end{lemma}

\begin{proof}
We want to show that it is surjective. Fix $t \in U_{4R}$. Consider again the sets $w_n$. Recall $\displaystyle{\psi^{-1}(t) = \bigcup_{n\geq0} W_{n}}$, where $W_n:= \psi^{-1}(t) \cap F^{-n}(D'_{(z,w)})$. 
We have defined $\Upsilon_t$ in $\psi^{-1}(t)$ by defining it on $W_0$ first, as in \eqref{Upsilondefinition}, and defining it recursively in $W_n$ with the relationship \eqref{Upsrecurse}. 
We will prove that $\displaystyle{\bigcup_{n \geq 0} \Upsilon_t(W_n) = \CC}$. For this purpose, we fix $n$. By definition, the image of $W_n$ by $\Upsilon_t$ is equal to the image of $\psi^{-1}(t-n) \cap D'_{(z,w)}$ by $\Upsilon_{t-n}$.
We first analyze the image of $\psi^{-1}(t-n) \cap D'_{(z,w)}$ by $\Upsilon_{t-n}$. This set is the image of $\{t-n\}\times D^{t-n}$ by $\xi_{t-n}$. Using Lemma \ref{Upsestimate} we get $\xi(\{t-n\}\times D^{t-n}) \supset T'_{R+2\log R, \frac{|t-n|}{2} - 2\log \frac{|t-n|}{2}} + \log(t-n)$. 
We see that, for a large $n$, $\Upsilon_t(W_n) \supset T'_{2R,n} + \log n$,  
Clearly we have $\bigcup_{n>N} T'_{2R,n} + \log n =\CC$, for any $N$.
This concludes the proof.
\end{proof}

We now have that the mapping $(\psi,\Upsilon)$ is a biholomorphism of $\psi^{-1}(U_{4R})$ onto $U_{4R} \times \CC$, and furthermore, for $n \in \NN$, $(\psi, \Upsilon \circ F^{n})$ is a biholomorphism of $\psi^{-1}(U_{4R} + n)$ onto $(U_{4R}+n) \times \CC$. 

So we have the function $\psi: \Omega \to \CC$, the open cover $\{\psi^{-1}(U_{4R} + n)\}_{n \in \NN}$ of $\Omega$, and the coordinates $(\psi,\Upsilon\circ F^n)$ on each $\psi^{-1}(U_{4R} + n)$ define on $\Omega$ a structure of locally trivial fiber bundle with base $\CC$ and fiber $\CC$.

\section{Biholomorphism to $\CC^2$}

We can define the following biholomorphic map to $\CC^2$ 
\begin{eqnarray*}
(\psi, \Upsilon): \Omega \to \CC^2
\end{eqnarray*}
and
\begin{eqnarray*}
(\psi, \Upsilon)(p) &= \left(\psi(F^n(p)) - n, \Upsilon(F^n(p))\right)
\end{eqnarray*}
for $p \in \Omega$. We have:
\begin{eqnarray*}
(\psi, \Upsilon)(F(p)) = \left(\psi(p) -1, \Upsilon(p)\right)
\end{eqnarray*}
for $p \in \Omega$.


\newpage
\begin{thebibliography}{ab-br-to}

\addtolength{\leftmargin}{4in} 
\setlength{\itemindent}{-0.2in} 

\bibitem[Ab1]{ab1}
	M. Abate, \emph{The residual index and the dynamics of holomorphic maps tangent to the identity}, Duke Math. J. \textbf{107} (2001), 173--207.
	
\bibitem[Ab2]{ab2}
	M. Abate, \emph{Discrete local holomorphic dynamics}. Proceedings of 13th {S}eminar on {A}nalysis and its {A}pplications (2003), 1--31.

\bibitem[ABT]{ab-br-to}
	M. Abate, F. Bracci and F. Tovena, \emph{Index theorems for holomorphic self-maps}, Ann. of Math. (2) \textbf{159} (2004), 819--864.

\bibitem[Bu-Fo]{bu-fo}
	G. T. Buzzard and F. Forstneric , \emph{An interpolation theorem for holomorphic automorphisms of {${\bf C}\sp n$}}, J. Geom. Anal. \textbf{10} (2000), 101--108.	
	
\bibitem[Fa]{fa2}
	P. Fatou, \emph{Substitutions analytiques et \'equations fonctionelles a deux variables}, An. Sc. Ec. N. (1924), 67--142.

\bibitem[Gr]{gr}
	M. Green, \emph{Holomorphic maps into complex projective spaces omitting hyperplanes}, Trans. Amer. Math. Soc. \textbf{169} (1972), 89--103.

\bibitem[Hak1]{hak1}
	M. Hakim, \emph{Analytic transformations of {$(\bold C\sp p,0)$} tangent to the identity}, Duke Math. J. \textbf{92} (1998), 403--428.

\bibitem[Mi]{mi}
	J. Milnor, \emph{Dynamics in one complex variable}, third ed., Annals of Mathematics Studies, vol. 160, Princeton University Press, Princeton, NJ, 2006.
		
\bibitem[Pe]{pe1}	
	H. Peters, \emph{Perturbed basins of attraction}, Math. Ann. \textbf{336} (2007), 1--13.

\bibitem[Pe-Wo]{pe-wo}	
	H. Peters and E. Wold, \emph{Non-autonomous basins of attraction and their boundaries}, J. Geom. Anal. \textbf{15} (2005), 123--136.
	
\bibitem[Ro-Ru]{ro-ru}
		J. P. Rosay and W. Rudin, \emph{Holomorphic maps from {${\bf C}\sp n$} to {${\bf C}\sp n$}}, Trans. Amer. Mat. Soc. \textbf{310} (1998), 47--86.

\bibitem[St]{st}
		B. Stens{\o}nes, \emph{Fatou-{B}ieberbach domains with {$C\sp \infty$}-smooth boundary}, Ann. of Math. (2) \textbf{145} (1997), 365--377.

\bibitem[Vi]{vi}
	L. Vivas, \emph{Remarks on automorphisms of $\CC^* \times \CC^*$ and their basins}, Complex Var. Elliptic Equ. \textbf{54} (2009), 401--408.
	
\bibitem[We1]{we}
	B. Weickert, \emph{Attracting basins for automorphisms of {${\bf C}\sp 2$}}, Invent. Math. \textbf{132} (1998), 581--605.

\bibitem[We2]{we2}
	B. Weickert, \emph{Automorphisms of {${\bf C}\sp n$}}.PhD Thesis, University of Michigan, (1997).

\bibitem[Wo1]{wo1}
    E. F. Wold, \emph{Fatou-{B}ieberbach domains}, Internat. J. Math. \textbf{16} (2005), 1119--1130.

\bibitem[Wo2]{wo2}
    E. F. Wold, \emph{A {F}atou-{B}ieberbach domain in {$\Bbb C\sp 2$} which is not {R}unge}, Math. Ann. \textbf{340} (2008), 775--780.
 

\end{thebibliography}
\end{document}